\documentclass[a4paper,reqno]{amsart}

\usepackage{amsmath,amsthm, amssymb}
\usepackage{mathrsfs}
\usepackage[shortlabels]{enumitem}
\usepackage{graphicx}
\usepackage[font=small]{caption}
\usepackage{xcolor}
\usepackage{url}

\newcommand{\N}{\mathbb{N}}
\newcommand{\R}{{\mathbb{R}}}
\newcommand{\C}{{\mathbb{C}}}

\newcommand{\dd}{{{\rm d}}}
\newcommand{\ii}{{\rm i}}

\newcommand{\cf}{\emph{cf.}}
\newcommand{\ie}{{\emph{i.e.}}}
\newcommand{\eg}{{\emph{e.g.}}}

\newcommand{\ov}{\overline}

\newcommand{\la}{\lambda}

\newcommand{\eps}{\varepsilon}

\renewcommand{\H}{{\mathcal{H}}}

\newcommand{\Dom}{{\operatorname{Dom}}}

\renewcommand{\Re}{\operatorname{Re}}
\renewcommand{\Im}{\operatorname{Im}}

\newcommand{\supp}{\operatorname{supp}}

\newcommand{\BigO}{\mathcal{O}}

\newcommand{\Dt}{-\frac{\dd^2}{\dd x^2}}
\newcommand{\Dtp}{\frac{\dd^2}{\dd x^2}}

\newcommand{\DD}{\Delta_{\rm D}}

\theoremstyle{plain}

\newtheorem{theorem}{Theorem}[section]
\newtheorem{lemma}[theorem]{Lemma}

\newtheorem{proposition}[theorem]{Proposition}
\newtheorem{corollary}[theorem]{Corollary}

\theoremstyle{definition}

\newcommand\cG{\mathcal G}
\newcommand\cH{\mathcal H}

\newcommand\cI{\mathcal I}

\newcommand\cL{\mathcal L}

\usepackage{mathtools}
\usepackage{url}
\mathtoolsset{showonlyrefs}

\newcommand{\eqskip}{ \vspace*{2mm}\\ }

\begin{document}

%\graphicspath{{../Figures/}} 

\title{The damped wave equation with singular damping}

\author{Pedro Freitas}
\address[Pedro Freitas]{
	Departamento de Matem\'{a}tica, Instituto Superior T\'{e}cnico, Universidade de Lisboa, Av. Rovisco Pais, 1049-001 Lisboa, Portugal
	\&
	Grupo de F\'{\i}sica Matem\'{a}tica, Faculdade de Ci\^{e}ncias, Universidade de Lisboa, Campo Grande, Edif\'{\i}cio C6,
	1749-016 Lisboa, Portugal }
%\curraddr{}
\email{psfreitas@fc.ul.pt}

\author{Nicolas Hefti}
\address[Nicolas Hefti]{
	Mathematical Institute, 
	University of Bern,
	Alpeneggstrasse~22,
	3012 Bern, Switzerland}
\email{nicolas.hefti@math.unibe.ch}%
\author{Petr Siegl}
\address[Petr Siegl]{
	School of Mathematics and Physics, Queen's University Belfast, University Road, BT7 1NN, Belfast, UK}
\email{p.siegl@qub.ac.uk}

\thanks{P.F. and P.S. were partially supported by FCT (Portugal) through project PTDC/MAT-CAL/4334/2014;
P.S. was partially supported by the \emph{Swiss National Science Foundation} Ambizione grant No.~PZ00P2\_154786
(until December 2017).}



\date{February 7, 2020}

\begin{abstract}
We analyze the spectral properties and peculiar behavior of solutions of a damped wave equation on a finite
interval with a singular damping of the form $\alpha/x$, $\alpha>0$. We establish the exponential stability
of the semigroup for all positive $\alpha$, and determine conditions for the spectrum to consist of a finite
number of eigenvalues. As a consequence, we fully characterize the set of initial conditions for which there
is extinction of solutions in finite time. Finally, we propose two open problems related to extremal decay rates
of solutions.
\end{abstract}

\maketitle

\section{Introduction}
We consider the linear damped wave equation 
\begin{equation}\label{DWE.intro}
\left\{
\begin{aligned}
u_{tt}(x,t)+\frac{2\alpha}xu_{t}(x,t)&=u_{xx}(x,t), \qquad x \in \cI:=(0,1),\quad t >0, \\
u(0,\cdot)=u(1,\cdot)&=0,\\
(u(\cdot,0),u_t(\cdot,0))&=(u_0,u_1) \in W_0^{1,2}(\cI) \times L^2(\cI)
\end{aligned}
\right.
\end{equation}
where $\alpha$ is a positive parameter.

The case $\alpha=1$ was studied by Castro and Cox in~\cite{Castro-2001-39}, where they
showed that in that instance every solution vanishes in finite time. More precisely, they proved
that $u(\cdot,t)\equiv0$ for $t>2$, irrespectively of the initial values. This analysis was
carried out within the context of the optimization of the spectral abscissa of the damped
wave equation, under the restriction that the damping term $a \in BV(\cI)$, \ie~a function of
bounded variation on $\cI$. To be more specific, if we consider the spectral problem associated
with~\eqref{DWE.intro} in the case of a general damping term $a \in BV(\cI)$, namely (we will
be more precise about this in Section~\ref{sec:spec}),
\begin{equation}\label{spec1}
\left\{
\begin{array}{ll}
 \lambda^{2} \phi + 2  \lambda a \phi = \phi_{xx}, & x\in\cI,\eqskip
 \phi(0) = \phi(1) = 0,
\end{array}
\right.
\end{equation}
the spectral abscissa is defined as the supremum of the real parts of the spectrum associated with~\eqref{spec1}. This is, in turn, related to the decay rate of solutions of~\eqref{DWE.intro}. Because of the bounded variation restriction, the approach in~\cite{Castro-2001-39} consisted in considering the family of functions $a_{n}(x) = 1/(x+1/n)$, thus showing that it was possible to make the spectral abscissa as large (negative) as possible within this class, by letting $n$ become sufficiently large.

Assuming that we do not restrict ourselves to functions of bounded variation, it becomes possible
to consider damping terms as those in~\eqref{DWE.intro} as a way of approaching the limiting case which 
yields finite-time extinction. 
It is then possible to show that the corresponding operator is well-defined and the associated semigroup is exponentially stable, \cf~Section~\ref{sec:spec} and Theorem~\ref{thm:semigroup}. Nonetheless, our aim is to understand the corresponding spectra and special features of the time evolution. In particular, we will show that what singles out the case of $\alpha=1$ is that the associated spectrum is, in fact, empty, and that
this is the only instance where this happens. This complements the well-known example with empty spectrum, the so-called complex Airy operator in 
$L^2(\R)$, namely
\begin{equation}
A = \Dt + \ii x, \qquad \Dom(A) = \{ f \in W^{2,2}(\R) \, : \, x f \in L^2(\R) \},
\end{equation}
\cf~for instance~\cite{Herbst-1979-64,Almog-2008-40}. However, unlike in the case of $A$ where the associated semigroup exhibits the super-exponential decay
\begin{equation}
\|e^{-t A} \| = e^{-\frac{t^3}{12}}, \qquad t>0,
\end{equation}
\cf~\cite[Sec.~14.3.2]{Helffer-2013-book}, here \emph{every solution vanishes in finite time}.

In this paper, we will show that the spectrum is empty only for $\alpha =1$, and that for non-integer positive $\alpha$, 
there are infinitely many eigenvalues with ``unusual'' asymptotic behavior, \cf~\eqref{la.k.asym}. Furthermore,
when $\alpha$ takes on a positive integer value $n>1$, all but $(n-1)$ eigenvalues disappear at infinity,
\cf~Theorem~\ref{thm:spectrum} for details. 
Moreover, for these positive integer values of $\alpha$, all but a finite-dimensional subspace of initial values, \cf~\eqref{Pn.cond}, lead to a vanishing solution in finite time, \cf~Theorem~\ref{thm:time};
the proof of these statements is based on a detailed analysis of the semigroup employing the Laplace transform and Paley-Wiener-Schwartz Theorem.

\section{Spectrum and exponential stability}
\label{sec:spec}

As usual, we rewrite \eqref{DWE.intro} as the first order system 
\begin{equation}\label{DWE.system}
\partial_t 
\begin{pmatrix}
u \\
v
\end{pmatrix}
=
G_0
\begin{pmatrix}
u \\
v
\end{pmatrix},
\end{equation}
where we start with an initial operator
\begin{equation}\label{G0.def}
G_0 := 	\begin{pmatrix}
0 & I \\[0.51mm]
\displaystyle \Dtp  & - \dfrac{2\alpha}{x}  
\end{pmatrix},
\quad 
\Dom(G_0):= \left( W^{1,2}_0(\cI) \cap W^{2,2}(\cI) \right)^2.
\end{equation}
It is known, see \eg~\cite{Freitas-2018-264}, that the closure 
\begin{equation}\label{G.def}
G:=\ov{G_0}
\end{equation}
in the space 
\begin{equation}\label{Hilb.space}
\begin{aligned}
\cH &:= W_0^{1,2}(\cI) \times L^2(\cI),
\\
\langle (\phi_1, \phi_2), (\psi_1, \psi_2) \rangle_\H &:= \langle \phi_1', \psi_1' \rangle_{L^2} +  \langle \phi_2, \psi_2 \rangle_{L^2},
\end{aligned}
\end{equation}
has convenient properties, namely $-G$ is m-accretive and thus it generates a contraction semigroup in $\cH$. Moreover, to analyze the spectrum of $G$, we rely on the spectral equivalence of $G$ with the associated quadratic operator function
\begin{equation}\label{T.def}
T(\la):= \Dt + \frac{2 \la  \alpha}{x} + \la^2, \quad \Dom(T(\la)) = W^{1,2}_0(\cI) \cap W^{2,2}(\cI),  \quad \la \in \C.
\end{equation}
As the damping term $2\alpha/x$ is relatively bounded with the bound $0$ with respect to the one dimensional Dirichlet Laplacian $-\DD$ in $L^2(\cI)$, the claims above follow from perturbation arguments and can be obtained from more general statements in \cite{Freitas-2018-264}.

\begin{proposition}\label{prop:basic}
Let $\alpha>0$, let $\cH$ be as in \eqref{Hilb.space} and $G$ as in \eqref{G.def}. Then 
\begin{enumerate}[\upshape (i)]
	\item  $G$ generates a contraction semigroup in $\cH$,
	\item  we have the spectral equivalence
	\begin{equation}\label{sp.equiv}
	\la \in \sigma(G) \qquad \Longleftrightarrow \qquad 0 \in \sigma(T(\la)),
	\end{equation}
	\item spectrum of $G$ is discrete, \ie~consisting only of isolated eigenvalues of finite algebraic multiplicity.
\end{enumerate}
\end{proposition}
\begin{proof}
It follows from the classical one-dimensional Hardy inequality on $\cI$, \ie~
\begin{equation}\label{Hardy.ineq}
\forall \psi \in W^{1,2}_0(\cI), \quad 
\int_\cI \frac{|\psi(x)|^2}{x^2} \, \dd x \leq 4 \int_\cI |\psi'(x)|^2 \, \dd x,
\end{equation}
that, for all $\psi \in W^{1,2}_0(\cI) \cap W^{2,2}(\cI)$ and each $\eps>0$,  
\begin{equation}
\begin{aligned}
|\alpha|^2 \int_\cI \frac{|\psi(x)|^2}{x^2}  \, \dd x
&\leq 
4 |\alpha|^2 \int_\cI |\psi'(x)|^2 \dd x
\leq
4 |\alpha|^2 \|\psi''\|_{L^2} \|\psi\|_{L^2}
\\
& \leq \eps \| \psi'' \|_{L^2}^2 + \frac{4 \alpha^4}{\eps} \|\psi\|_{L^2}^2.
\end{aligned}
\end{equation}
Thus the damping $a(x) = a_{\rm s}(x)= \alpha/x$ satisfies \cite[Asm.~I]{Freitas-2018-264} and also (trivially) \cite[Asm.~II]{Freitas-2018-264}, so the claims follow from \cite[Thm.~2.2, 3.2]{Freitas-2018-264}; we note that the spectral equivalence from \cite[Thm.~3.2]{Freitas-2018-264} can be extended to all $\la \in \C$ since $\alpha/x$ is relatively bounded with the bound $0$ with respect to $-\DD$ both in the operator and form sense. 
\end{proof}
The spectral equivalence \eqref{sp.equiv} and known facts on the confluent hypergeometric functions allow us to 
describe the eigenvalues and eigenfunctions of $G$ quite precisely.
\begin{theorem}\label{thm:spectrum}
Let $\alpha>0$ and let $G$ be as in \eqref{G.def}. Then 
\begin{enumerate}[\upshape (i)]
\item the eigenvalues $\la$ of $G$ satisfy
\begin{equation}\label{EV.eq}
\la \in \sigma(G) \qquad \Longleftrightarrow \qquad M(1-\alpha,2,-2 \la) =0,
\end{equation}
where $M$ is the Kummer function, \cf~\cite[Sec.~13]{DLMF}, and the corresponding eigenvectors can be selected as $(f_\la, \la f_\la)^t$ with
\begin{equation}\label{f.la.def}
f_\la(x) = x e^{\la x} M(1-\alpha,2,-2 \la x),
\end{equation}
\item if $\alpha \notin \N$, then $\sigma(G)$ contains exactly $\lceil \alpha-1 \rceil$ negative eigenvalues and infinitely many complex conjugated (non-real) eigenvalues $\{\la_k^{(\alpha)}\}$ satisfying the asymptotic relation (for large $|\la_k^{(\alpha)}|$)
\begin{equation}\label{la.k.asym}
\begin{aligned}
\la_k^{(\alpha)} &= \mp \frac{2k+1-\alpha}{2} \pi \ii 
- \frac{1}{2} \log 
\left(
- \frac{\Gamma(1-\alpha)}{\Gamma(2+\alpha)}(\pm 2k\pi \ii)^{2\alpha}
\right) 
\\
& \quad + \BigO(k^{-1} \log k), \quad k \to \infty,
\end{aligned}
\end{equation}

\item \label{thm:sp.iii} if $\alpha =n+1$, $n \in \N_0$, then 
\begin{equation}
\sigma(G)=\{\mu_k^{(n)}\}_{k=1}^{n} \subset (-\infty,0),
\end{equation}
\ie~the spectrum of $G$ consists of exactly $n$ negative eigenvalues. The latter are determined by
\begin{equation}
L^{(1)}_n(-2 \mu) = 0,
\end{equation}
where $L_n^{(1)}$ are associated Laguerre polynomials, \cf~\cite[Eq.~18.5.12]{DLMF}, and the corresponding eigenvectors can be selected as $(f_k^{(n)}, \mu_k^{(n)} f_k^{(n)})^t$ with
\begin{equation}\label{f.k.n.def}
f_k^{(n)}(x) := x \exp (\mu_k^{(n)} x ) L_n^{(1)}(-2 \mu_k^{(n)} x).
\end{equation}
\end{enumerate}
\end{theorem}
\begin{proof}
Relying on the spectral equivalence \eqref{sp.equiv} and on that the spectrum of $G$ is discrete, it suffices to analyze the non-linear spectral
problem for the associated operator function $T(\la)$, \ie~
\begin{equation}\label{Tla.EV}
\left\{
\begin{aligned}
-f''(x) + \frac{2\la \alpha}{x} f(x) + \la^2 f(x) &= 0,
\\
f(0) = f(1) &= 0.
\end{aligned}
\right.
\end{equation}
If eigenpairs $(\la,f_\la)$ are found, then these $\la$'s are eigenvalues of $G$ and it can be easily checked that the corresponding eigenvectors can be selected as $(f_\la,\la f_\la)^t$.

The substitutions $f(x) = x e^{\la x} v(x)$ and $\xi=-2\la x$ bring the differential equation in \eqref{Tla.EV} to the Kummer equation
\begin{equation}
\xi v_{\xi \xi} + (2-\xi) v_\xi -(1-\alpha) v =0.
\end{equation}
If $\alpha \notin \N$, then the general solution read (with $C_1$, $C_2 \in \C$),
\begin{equation}
C_1 M(1-\alpha,2,\xi) + C_2 U(1-\alpha,2,\xi),
\end{equation}
\cf~\cite[Sec.~13.2]{DLMF} for the definition and properties of Kummer functions~$M$~and~$U$. The boundary condition $f(0)=0$ in \eqref{Tla.EV} and the behavior of $U(1-\alpha,2, \cdot)$ at $0$ imply that $C_2=0$ and the second boundary condition $f(1)=0$
yields the eigenvalue equation in \eqref{EV.eq}. The claims on eigenvalues are based on known facts on zeros of the Kummer function $M$, \cf~\cite[Sec.~13.9]{DLMF}.

If $\alpha \in \N$, then the solution $U$ must be replaced, \cf~\cite[Eq.~13.2.28]{DLMF}, nonetheless, the same conclusion is obtained. Namely, $C_2=0$ and the
eigenvalue equation in \eqref{EV.eq} remains valid. Moreover, for $\alpha = n+1$, $n \in \N_0$, the Kummer function reduces to the associated Laguerre polynomial $L_n^{(1)}$ and the claim \ref{thm:sp.iii} follows.
\end{proof}

The eigenvalues of $G$ depending on $\alpha$ are illustrated in Figure~\ref{fig:EV}. 

\begin{figure}[htb!]
	\includegraphics[width=0.45 \textwidth]{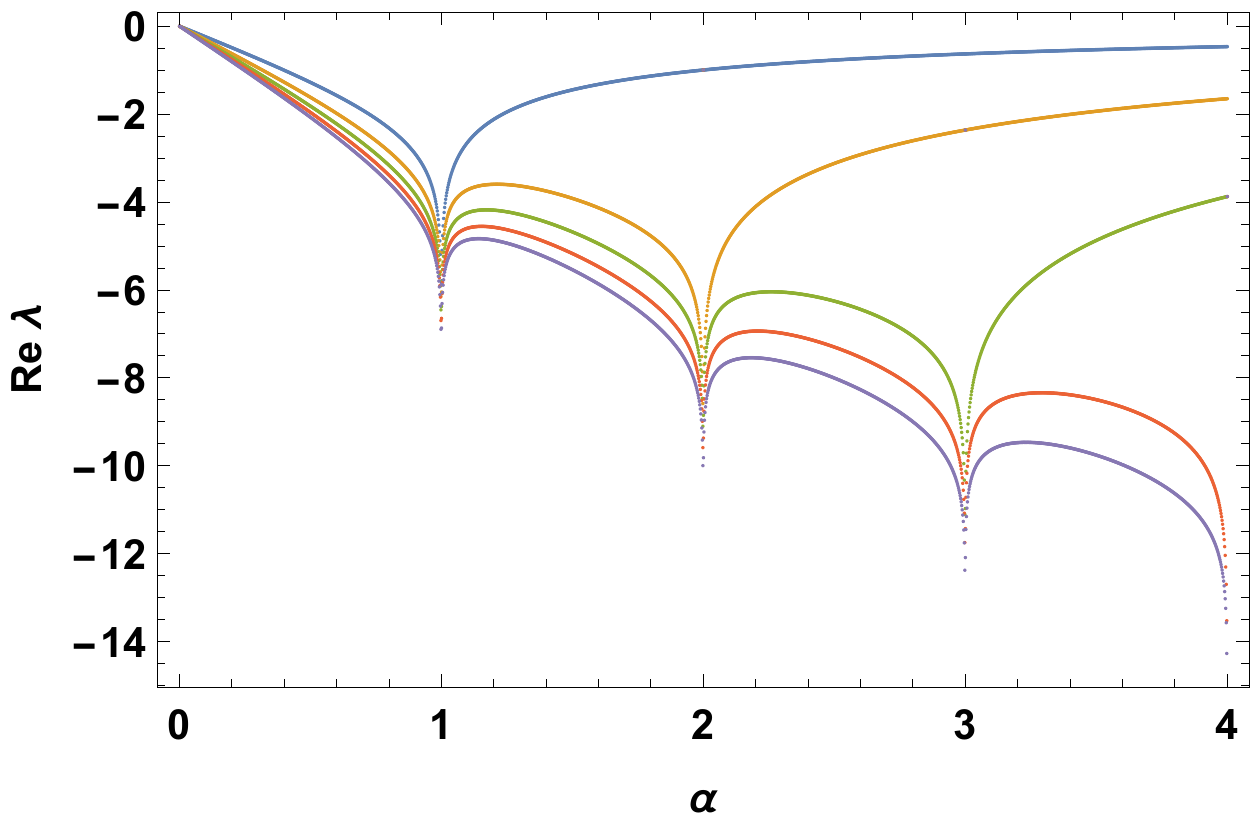}
	\hfill
		\includegraphics[width=0.45 \textwidth]{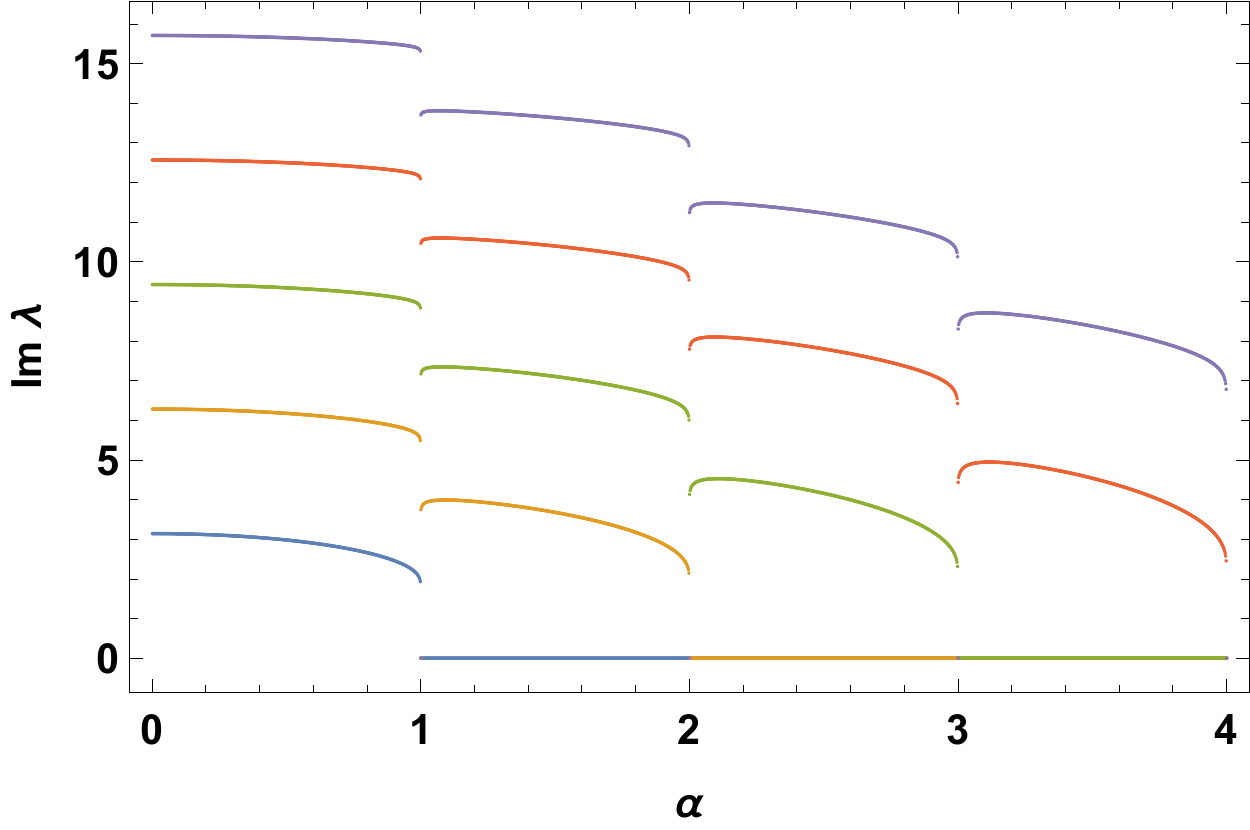}
	\caption{Real and imaginary parts of five eigenvalues of $G$ as a function of $\alpha$.}
	\label{fig:EV}
\end{figure}

Finally, we show the exponential stability of the semigroup generated by $G$ and the validity of a spectral bound for all sufficiently large $\alpha>0$ and all $\alpha = n+1$, $n \in \N$, \cf~\cite[Chap.~IV, V]{Engel-Nagel-book} for details on stability notions.
\begin{theorem}\label{thm:semigroup}
Let $\alpha > 0$ and let $e^{tG}$ be the semigroup generated by $G$ from \eqref{G.def}. Then the following hold: 
\begin{enumerate}[\upshape (i)]
\item \label{expdeacy.1} the semigroup $e^{tG}$ is uniformly exponentially stable, \ie~
\begin{equation}\label{exp.stab.def}
\exists \epsilon > 0, \quad \lim_{t \to +\infty} e^{\epsilon t} \| e^{tG} \| = 0,
\end{equation}
\item\label{expdeacy.2} for sufficiently large $\alpha \notin \N$, the spectral growth bound condition holds,~\ie~
\begin{equation}\label{spec.bound}
s(G):=\sup_{\lambda \in \sigma(G)} \Re{\lambda} = \inf{\{w \in \R \mid \lim_{t \to \infty}  e^{-w t}\| e^{tG} \| = 0\}} := \omega_0(G), 
\end{equation}
\item\label{expdeacy.3} for every $\alpha= n + 1 \in \N$, the spectral growth bound condition holds.
\end{enumerate}
\end{theorem}
\begin{proof}
The idea is to estimate the resolvent norm on vertical lines and deduce the claims using Gearhart-Prüss Theorem, \cf~\cite[Thm.~V.1.11]{Engel-Nagel-book}, its corollary, \cf~\cite[Ex.~V.1.13]{Engel-Nagel-book},
\begin{equation}\label{omega_0}
\omega_0(G)= \inf{\{ \Re z > s(G) \, : \, \sup_{\Im z \in \R} \| (G-z)^{-1} \| < \infty\}} 
\end{equation}
and results on the localization of the eigenvalues of $G$.

For the resolvent estimate, we follow the idea of \cite[Ex.~4.6]{Schnaubelt-2018-LNEE} or \cite[Thm.~VI.3.18]{Engel-Nagel-book}.
Let $\tau:=- \sigma + \ii \eta \in \C$ with $\sigma > 0$, $\eta \in \R$. To obtain an upper bound for $\|(G-\tau)^{-1} \|$, we estimate $\|(G -\tau) w \| $
from below for an arbitrary $w=(u,v) \in \Dom(G_0)$ with 
\begin{equation}\label{w.norm}
\| w \|^2_\cH = \|u'\|^2_{L^2} + \| v \|^2_{L^2} = 1.
\end{equation}
Using the latter and  Cauchy-Schwarz inequality, we get (with $a(x):=\frac{2 \alpha}{x}$)
\begin{equation}\label{stab.ineq1}
\begin{aligned}
\|(G-\tau) w \|_\cH & \geq |\langle (G-\tau) w,w \rangle_\cH |
\\ & = | \langle -\tau u' + v', u' \rangle_{L^2} + \langle u'' - (a+\tau) v, v \rangle_{L^2} |\\
&= | \sigma - \| \sqrt{a} v \|_{L^2}^2 - \ii( \eta + 2 \Im \langle u',v' \rangle_{L^2})   |
\\
& \geq \max \{ 2 \alpha \| v \|_{L^2}^2 - \sigma,  | \eta + 2 \Im \langle u',v' \rangle_{L^2}  | \}. 
\end{aligned}
\end{equation}
Thus in particular, 
\begin{equation}\label{Imu'v'}
|\Im \langle u',v'\rangle_{L^2}| \leq \frac 12 \|(G-\tau) w \|_\cH + \frac{|\eta|}{2}.
\end{equation}
On the other hand, by Cauchy-Schwarz inequality and since $\|u'\|_{L^2}\leq 1$, we get
\begin{equation}\label{stab.ineq2}
\begin{aligned}
\|(G-\tau) w \|_\cH & \geq  \|-\tau u' + v' \|_{L^2}
\geq | \langle -\tau u' + v', u' \rangle_{L^2}| 
\\
&\geq |  \Im \langle - \ii \eta u' + v', u' \rangle_{L^2}| 
 \geq |\eta| \| u' \|_{L^2}^2 - |\Im \langle u',v' \rangle_{L^2}|,
\end{aligned}
\end{equation}
thus we get from \eqref{Imu'v'} that 
\begin{equation}\label{u'.est}
2 \| u' \|_{L^2}^2 \leq \frac 3{|\eta|} \|(G-\tau) w \|_\cH + 1.
\end{equation}
Hence combining the first resulting inequality in \eqref{stab.ineq1}, \eqref{w.norm} and \eqref{u'.est}, we obtain
\begin{equation}\label{res.est}
\|(G-\tau) w \|_\cH \geq  \frac{ \alpha - \sigma}{1+\frac{ 3 \alpha}{\eta}} 
\end{equation}
which holds for all normalized $w \in \Dom(G_0)$. Since $\Dom(G_0)$ is a core of $G$, \eqref{res.est} holds also for all normalized $w \in \Dom(G)$. Thus
the resolvent norm is bounded on all vertical lines (up to possible points in the spectrum) with $\sigma < \alpha$. 

\ref{expdeacy.1} Since zero is not an eigenvalue for $G$ for any $\alpha>0$, the uniform exponential stability follows from Gearhart-Prüss Theorem by
plugging in $\sigma = 0$.  

\ref{expdeacy.2} It follows from an asymptotic formula for real zeros of Kummer function, \cf~\cite[Eq.~13.9.8]{DLMF}, that there is a negative eigenvalue
of $G$ which tends to $0$ as $\alpha \to +\infty$, thus the claim follows from \eqref{omega_0} and \eqref{res.est}.

\ref{expdeacy.3} Let $\alpha = n + 1 \in \N$. The largest negative eigenvalue $\mu_0^{(n)}$ of $G$, \ie~the largest zero of $L^{(1)}_n(-2 \mu)$ satisfies, \cf~\cite{Gupta-2007-8}, 
\begin{equation}
|\mu_0^{(n)}| \leq \frac{3}{2+n} , \qquad n \geq 1.  
\end{equation}
Hence, $s(G) \geq - 1$ for all $n \in \N$ and thus the claim follows by \eqref{omega_0} and \eqref{res.est}. 
\end{proof}
We remark that a numerical computation of the position of the right-most eigenvalue of $G$ as a function of $\alpha$ suggests that the spectral bound
condition,~\cf~\eqref{spec.bound}, holds at least for $\alpha > 1.48$.

\section{Time evolution for $\alpha \in \N$}
We investigate a long-time behavior of the solutions $u$ of \eqref{DWE.intro} for $\alpha \in \N$. Using the Laplace transform with respect to the variable $t$
\begin{equation}
U(x,\tau):=\cL_t[u(x,t)](\tau)=\int_0^{\infty} e^{- \tau t}u(x,t)\;\dd t, \qquad \Re \tau > 0, 
\end{equation} 
the solution $u$ for $\alpha=1$ was calculated in \cite{Castro-2001-39} and thereby proved that $u(\cdot,t)$ vanishes for $t>2$. It is clear that the very special behavior for $\alpha=1$ cannot be valid for other integer $\alpha$'s and general initial values since the point spectrum of $G$ is non-empty and thus the standing wave solutions appear. In more detail, for $\alpha = n+1$, $n \in \N$, we get solutions
\begin{equation}
u_k^{(n)}(x,t) = \exp (\mu_k^{(n)} t) f_k^{(n)}(x), \qquad x \in \cI, \ t >0, \ k = 0,\dots,n,
\end{equation}
see Theorem~\ref{thm:spectrum}, which clearly do not vanish in a finite time. Nonetheless, employing  the Paley-Wiener-Schwartz Theorem, we show that the peculiar feature $u(\cdot,t)=0$ for $t>2$ is preserved if the initial values $(u_0,u_1) \in \cH$ are chosen more carefully, \cf~Theorem~\ref{thm:time} below. As the spectrum of $G$ for $\alpha=n+1$ contains only $n$ simple eigenvalues, it is natural that the condition on the initial value is that
\begin{equation}\label{Pn.cond}
P_n \begin{pmatrix}
u_0 \\ u_1
\end{pmatrix} = 0,
\end{equation}
where $P_n$ is the spectral projection on the invariant subspace of $G$ associated with eigenvalues $\{\mu_k\}_{k=1}^n$. We note that the condition \eqref{Pn.cond} can be expressed in terms of $(f_k^{(n)},-\mu_k^{(n)} f_k^{(n)})$, $k=1,\dots,n$, which are the eigenfunctions of $G^*$, \cf~the condition \eqref{og.asm} in Theorem~\ref{thm:time} below.

\begin{theorem}\label{thm:time}
Let $n\in \N$ and $\alpha=n+1$ and $f_k^{(n)}$ with $k=1,\dots,n$ be as in \eqref{f.k.n.def} in Theorem \ref{thm:spectrum}.
Then the solution of \eqref{DWE.intro} vanishes in finite time if and only if 
\begin{equation}\label{og.asm}
\langle (u_0,u_1),(f_k^{(n)},-\mu_k^{(n)} f_k^{(n)}) \rangle_{\cH}=0, \qquad k = 0,\dots,n. 
\end{equation}
In such a case the solution vanishes for $t > 2$. 
\end{theorem}

The proof of this result is given in the next section.

\subsection{Technical steps and the proof of Theorem~\ref{thm:time}}
\label{subsec:proof.time}

First, we set
\begin{equation}\label{r.def}
r(x,\tau):= \tau u_0(x)+u_1(x)+\frac{2 (n+1) u_0(x)}{x};
\end{equation}
notice that by the one-dimensional version of Hardy inequality, \cf~\eqref{Hardy.ineq}, we have $r(\cdot,\tau) \in L^2(\cI)$ for $(u_0, u_1) \in \cH$.

Second, we apply $\cL_t$ on \eqref{DWE.intro} and obtain (with $\Re \tau >0$)
\begin{equation}\label{prob:1.L}
\left\{
\begin{aligned}
-U_{xx}(x,\tau)+\tau^2 U(x, \tau) + \frac{2 (n+1) \tau }{x}U(x,\tau) & = r(x,\tau), \qquad x \in \cI,\\
U(0,\tau)=U(1,\tau)&=0,
\end{aligned}
\right.
\end{equation}
which is a resolvent equation for $T(\tau)$, \ie~
\begin{equation}
T(\tau) U= r.
\end{equation}
This is a Sturm-Liouville problem in $L^2(\cI)$ leading to the Laguerre equation of order $n$. To   
express the Green's function, \ie~the integral kernel of the operator $T(\tau)^{-1}$, we need the associated Laguerre polynomials $L_n^{(1)}$, \cf~\cite[Chap.~18.3]{DLMF}, and polynomials 
$P_n$ 
\begin{equation}\label{P.def}
P_n(x):= \sum_{k=0}^{n}  (-1)^k \binom{n+1}{k+1} \sum_{m=0}^{k} \frac{(k-m)!}{k! }x^m, 	
\end{equation}
\cf~\cite{Parke-2016-11}, in the notation of which $P_n(x) = P(n,1,x) / n!$. 
%To simplify notations we introduce functions (reducing to the eigenfunctions $f_k^{(n)}$ for $\tau = \mu_k^{(n)}$)
%%
%\begin{equation}
%\varphi_n(x, \tau):=x \exp(\tau x) L_n^{(1)}(-2 \tau x).
%\end{equation}

\begin{lemma}\label{lem:U}
Let $(u_0, u_1) \in \cH$, let $\alpha = n + 1 \in \N$ and let $r$ be as in \eqref{r.def}. Then, for all $\tau \in \C \setminus \{\mu_k^{(n)}\}_{k=1}^{n}$, \cf~Theorem~\ref{thm:spectrum} \ref{thm:sp.iii},  the solution of \eqref{prob:1.L} reads

\begin{equation}\label{U.sol}
\begin{aligned}
U(x,\tau) &= U_1(x,\tau) + U_2(x,\tau),
\\
U_1(x,\tau)& = \int_0^x \cG_1(x,y,\tau) r(y,\tau) \, \dd y +  \int_x^1 \cG_1(y,x,\tau) r(y,\tau) \, \dd y,
\\
U_2(x,\tau) & =  \sum_{k=1}^n \frac{a_k^{(n)}}{\tau - \mu_k^{(n)}}
\int_0^1 \cG_2(x,y,\tau) r(y,\tau) \, \dd y,& x \in \cI,
\end{aligned}
\end{equation}
with ($x,y \in \cI$)
\begin{equation}
\begin{aligned}\label{Quotientterm}
\cG_1(x,y,\tau) & =  \frac{1}{n+1}  y e^{\tau y}L^{(1)}_n(- 2\tau y) \times
\\
& \quad \Bigg[
e^{-\tau x} P_n(-2 \tau x) 
- x e^{\tau x}L^{(1)}_n(- 2\tau x)
\Big(
2\tau \int_x^1 \frac{e^{-2 \tau s}}{s} \; \dd s + e^{-2 \tau}
\Big) 
\Bigg],
\\
\cG_2(x,y,\tau) & = - \frac{1}{n+1}
x y  e^{\tau (x+y-2)} L^{(1)}_n(- 2\tau x) L^{(1)}_n(- 2\tau y).
\end{aligned}
\end{equation}
Here $P_n$ are as in \eqref{P.def}, $L_n^{(1)}$ are the associated Laguerre polynomials, 
the numbers $\{a_k^{(n)}\}_{k=1}^n$ are defined by the equation

\begin{equation}\label{ak.def}
\frac{ P_n (-2\tau)}{L^{(1)}_n(-2\tau)} = 	1 +\sum_{k=1}^n \frac{a_k^{(n)}}{\tau - \mu_k^{(n)}}
\end{equation}
and are all non-zero.
\end{lemma}
\begin{proof}
We start with assuming that $\Re \tau > 0$ and find the Green's function $\cG$ in a standard way. The formula \eqref{U.sol} follows by manipulating
\begin{equation}\label{Gr.formula}
\int_0^1 \cG(x,y,\tau) r(y,\tau) \, \dd y.
\end{equation}
We want to obtain two linearly independent solutions $y_L$ and $y_R$ of the homogeneous equation associated with \eqref{prob:1.L}, such that $y_L(0)=0$ and $y_R(1)=0$. The substitutions 
\begin{equation}
U(x,\tau)=xe^{\tau x} w(x), 
\end{equation} 
and $\xi=- 2 \tau x$, \ie~$w(x)=v(\xi)$, lead to the Laguerre equation of order $n$
\begin{equation}\label{L.ODE}
\xi v_{\xi \xi}+(2-\xi)v_{\xi}+ n v=0. 
\end{equation}
Two linear independent solutions of \eqref{L.ODE} are given by the Laguerre polynomial $L^{(1)}_n(\xi)$, \cf~\cite[Eq.~18.5.12]{DLMF} and, following \cite{Parke-2016-11} with using that $P_n(0) = 1$, by 
\begin{equation}\label{Quotientterm3}
v(\xi)=P_n(\xi) \frac{e^{\xi}}{\xi} + L^{(1)}_n(\xi) E_1(-\xi), 
\end{equation}
where $E_1$ denotes the integral exponential function 
\begin{equation}
E_1(\zeta) = \int_1^\infty \frac{e^{-t \zeta}}{t} \dd t, \qquad \Re \zeta > 0;	
\end{equation}
we remark that $\Re(-\xi) > 0$ due to $\xi= - 2 \tau x$ and $\Re \tau >0$ by assumption.
Hence we choose the sought solutions as
\begin{equation}
\begin{aligned}
y_L(x)  :=xe^{\tau x}L^{(1)}_n(- 2\tau x),
\quad 
y_R(x) :=y_L(x)e^{\tau}v(-2\tau) - xe^{\tau x}y_L(1)v(-2\tau x). 
\end{aligned}
\end{equation}
Using that $E_1'(x)  = -\exp(-x)/x$, 
we obtain the Wronskian of $y_L$ and $y_R$, 
\begin{equation}
W[y_L,y_R](0) =
\frac{(n+1) L^{(1)}_n(-2 \tau)}{2 \tau e^{-\tau}}.
\end{equation}
Hence the Green's function $\cG$ is given by
\begin{equation}
\cG(x,y,\tau) = \frac{2 \tau e^{-\tau}}{(n+1)L^{(1)}_n(-2 \tau)} 
y_L(x)y_R(y), \qquad x<y.
\end{equation}
We arrive at \eqref{U.sol} by straightforward manipulations and employing \eqref{ak.def}; notice that \eqref{ak.def} can be satisfied since the highest order coefficients of $L_n^{(1)}$ and $P_n$ coincide, \cf~\eqref{P.def} and \cite[Eq.~18.5.12]{DLMF}.

Finally, one can clearly extend the formula \eqref{U.sol} to $\tau \in \C \setminus \{\mu_k^{(n)}\}_{k=1}^{n}$. Moreover, the numbers $\{a_k^{(n)}\}_{k=1}^n$ must be all non-zero since otherwise we can extend the formula \eqref{U.sol} even to some $\tau \in \{\mu_k^{(n)}\}_{k=1}^{n}$, which is a contradiction with $0 \in \sigma(T(\mu_k^{(n)}))$, $k=1, \dots, n$, \cf~\eqref{sp.equiv} and Theorem~\ref{thm:spectrum}.
\end{proof}

Before we proceed we want to mention that we can prove the finite time extinction result of \cite{Castro-2001-39} already with Lemma \ref{lem:U} and without any
formula for $u$. 

\begin{corollary}[{\cite[Thm.~1.2]{Castro-2001-39}}]
Every solution $u$ of \eqref{DWE.intro} with $\alpha =1$ vanishes for $t > 2$. 
\end{corollary}

\begin{proof}
By Lemma~\ref{lem:U} with $n=0$ we get a formula for $U$, the Laplace transform of $u$. After some manipulations, this formula reads as
\begin{align}
U(x,\tau)&=x\int\limits_x^1\frac{u_0(r)}re^{ \tau (x-r)}\;\mathrm{d}r \\
&\quad+x\int\limits_x^1\frac1{r^2}\int\limits_0^r (u_0(s) -  su'_0(s) + su_1(s)) e^{ \tau (x-2r+s)}\;\mathrm{d}s\;\mathrm{d}r.
\end{align}
For fixed $x \in \cI$ clearly $U(x,\cdot)$ extends to an entire function on $\C$. Moreover, the following estimate holds (uniformly in $x \in [0,1]$) for some constant $C > 0$:
\begin{equation}
\begin{aligned}
\vert U(x,\tau) \vert \leq C e^{2\vert \Re{\tau} \vert}; 
\end{aligned}
\end{equation}
to check that the estimate is uniform in $x \in [0,1]$, one uses that $(u_0,u_1) \in \cH$, \cf~\eqref{Hilb.space}, and Hardy inequality \eqref{Hardy.ineq}. Hence, by the Paley-Wiener-Schwartz theorem, \cf~\cite[Thm.~7.3.1]{Hormander-book-I}, it follows that the inverse Laplace transform $u_1(x,\cdot)$ of $U_1(x,\cdot)$ is a distribution with a compact support satisfying $\supp{u(x,\cdot)} \subset [-2,2]$. 
\end{proof}

To obtain the condition \eqref{og.asm} in Theorem~\ref{thm:time} we need the following observation.
\begin{lemma}\label{lem:condition}
Let $n \in \N$ and let $\{f_k^{(n)}\}_{k=1}^n$ be as in \eqref{f.k.n.def} in Theorem \ref{thm:spectrum} and let $r$ be as in \eqref{r.def}. Then 
\begin{equation}
\langle (u_0,u_1),(f_k^{(n)},-\mu_k^{(n)} f_k^{(n)}) \rangle_{\cH}
%=\int_0^1 x L^{(1)}_n \left(-2x\mu_j^{(n+1)} \right) e^{x\mu_j^{(n+1)}}r\left(x,\mu_j^{(n+1)}\right)\;\dd x.
= \langle r(\cdot,\mu_k^{(n)}),f_k^{(n)} \rangle_{L^2}.
\end{equation}
\end{lemma}
\begin{proof}
Before we start with the proof we mention that the Laguerre polynomials satisfy the relation 
\begin{equation}\label{recurrence}
L^{(1)}_n(z)=L^{(0)}_n(z)-\frac{e^{z}}{(n-1)!}\frac{\mathrm{d}^{n}}{\mathrm{d}z^{n}}( e^{-z}z^{n-1}), \qquad n \in \N.
\end{equation}
This can be easily seen by using the contiguous relation \cite[Eq.~18.9.13]{DLMF} and the differential representation of the Laguerre polynomials \cite[Eq.~18.5.5]{DLMF}
\begin{equation}
L^{(\beta)}_n(x)=\frac{e^x}{n!x^\beta}\frac{\mathrm{d}^n}{\mathrm{d}x^n}\left(e^{-x}x^{n+\beta}\right), \qquad n \in \N, \quad \beta>-1.
\end{equation}
To simplify notations, let $\mu:=\mu_k^{(n)}$ and $z:=-2x\mu$. Integrating by parts, we get
\begin{equation}
\begin{aligned}
&-\frac1{\mu} \int_0^1 u_0'(x) \frac{\mathrm{d}}{\mathrm{d}x} 
\left(x L^{(1)}_n \left(- 2\mu x \right) e^{x \mu}\right)\;\mathrm{d}x
\\
&=\frac1{\mu} \int_0^{-2 \mu} u_0\left(-\frac{z}{2 \mu}\right) \frac{\mathrm{d^2}}{\mathrm{d}z^2
\left( e^\frac z2 \frac{\mathrm{d}^n}{\mathrm{d}z^n}(e^{-z}z^{n+1})\right)\;\mathrm{d}z }
\\
&=\frac1{\mu n!}\int_0^{-2 \mu} u_0 \left(-\frac{z}{2 \mu}\right) e^{\frac z2}
\left(
\frac 14  \frac{\mathrm{d}^n}{\mathrm{d}z^n} + 
\frac{\mathrm{d}^{n+1}}{\mathrm{d}z^{n+1}} +
\frac{\mathrm{d}^{n+2}}{\mathrm{d}z^{n+2}} 
\right)
(e^{-z}z^{n+1})\;\mathrm{d}z
\\
&=\frac{1}{\mu }\int_0^{-2 \mu} u_0 \left(-\frac{z}{2 \mu}\right)  e^{-\frac z2} 
\Big(
\frac z4 
 L^{(1)}_n(z) - (n+1) L^{(0)}_n(z) 
\\
& \qquad \qquad \qquad \qquad \qquad \qquad \quad +\frac{(n+1)}{(n-1)!} e^z \frac{\mathrm{d}^{n}}{\mathrm{d}z^{n}}( e^{-z}z^{n-1})
\Big)
\;\mathrm{d}z
\end{aligned}
\end{equation}
and by using formula \eqref{recurrence}, we obtain the desired result
\begin{align}
&\frac{1}{ \mu}\int_0^{-2 \mu} u_0\left(-\frac{z}{2 \mu}\right)  e^{-\frac z2} \left( \frac z4 L^{(1)}_n(z) - (n+1) L^{(1)}_n(z)\right) \;\mathrm{d}z
\\
&=\int_0^{1}  x e^{\mu x}L^{(1)}_n(- 2\mu x)\left( u_0(x) \mu + \frac{2(n+1)u_0(x)}{ x}\right) \;\mathrm{d}x. \qedhere
\end{align}
\end{proof}

\begin{proof}[Proof of Theorem~\ref{thm:time}]
We want to find the inverse Laplace transform of \eqref{U.sol}, \ie~
we search for
\begin{equation}
u_j(x,t):=\mathcal{L}^{-1}_{\tau}[U_j(x,\tau)](x,t), \qquad j=1,2.
\end{equation}
Technically, we work with the distributional transform, \cf~\cite[Sec.~9]{Vladimirov-2002} or \cite[Chap.~VII]{Hormander-book-I}; notice a different convention there, one works with $z= \ii \tau$. The inverse transform is found explicitly for $u_2$ and $t>2$, which, based on Lemma~\ref{lem:condition}, yields the condition \eqref{og.asm} and possibly non-vanishing part of the solution. Although it is possible to find $u_1$ explicitly as well, we choose to show that $u_1$ vanishes for $t>2$ in a more elegant way relying on the Paley-Wiener-Schwartz theorem. 

To prove that $u_1(\cdot,t) = 0$ for $t>2$, observe that for every fixed $x \in [0,1]$, the function $U_1(x,\tau)$ is entire in $\tau$, \cf~Lemma~\ref{lem:U}. Moreover, straightforward estimates show that there exists a constant $C>0$ such that (uniformly in $x \in [0,1]$)
\begin{equation}
|U_1(x,\tau)| \leq C (1+|\tau|)^{2(n+1)} e^{2 |\Re \tau|}, \quad \tau \in \C. 
\end{equation}
Hence the Paley-Wiener-Schwartz theorem, \cf~\cite[Thm.~7.3.1]{Hormander-book-I}, yields that the inverse Laplace transform $u_1(x,\cdot)$ of $U_1(x,\cdot)$ is a distribution with a compact support satisfying $\supp{u(x,\cdot)} \subset [-2,2]$. 

To derive the formula for $u_2$, observe that, since $-2 \leq x+y-2 \leq 0$, we have for $\mu <0$ that
($\vartheta$ denotes the Heaviside step function)
\begin{equation}
\cL_t[\vartheta(t+x+y-2)e^{\mu(t+x+y-2)}](\tau)=e^{\tau(x+y-2)}\cL_t[\vartheta(t)e^{t\mu}](\tau)
=\frac{e^{\tau(x+y-2)}}{\tau - \mu}.
\end{equation}
If $t>2$, then $\vartheta(t+x+y-2)=1$, thus using the rules for the distributional Laplace transform, we obtain
\begin{equation}\label{G3}
\begin{aligned}
&-(n+1) \, \cL^{-1}_{\tau}\left[ \cG_2(y,x,\tau)r(y,\tau)(\tau-\mu)^{-1}\right](t)
\\
&\quad = \, x L^{(1)}_n(-2x\partial_t) \, r(y,\partial_t)  y L^{(1)}_n(-2y\partial_t)e^{\mu(t+x+y-2)} 
\\
&\quad = \, x L^{(1)}_n (-2 \mu x ) \, r (y,\mu) y L^{(1)}_n(-2 \mu y)e^{\mu(t+x+y-2)}, \quad t>2.
\end{aligned}
\end{equation}
Hence, applying Lemma~\ref{lem:condition}, we get for $t>2$
\begin{equation}
\begin{aligned}
 u_2(x,t) & = - \sum_{k=1}^n \frac{a_k}{n+1}  e^{\mu_k^{(n)} (t-2)} f_k^{(n)}(x) \langle (u_0,u_1),(f_k^{(n)},-\mu_k^{(n)} f_k^{(n)}) \rangle_{\cH}.
\end{aligned}
\end{equation}
From the linear independence of $\exp(\mu_k^{(n)} t) f_k^{(n)}(x) $ and since $a_k \neq 0$,  $j=1,..,n$,  we conclude that $u_2(\cdot,t)=0$ if and only if \eqref{og.asm} is satisfied.
\end{proof}

\section{Further remarks and two open problems}

Changing $t$ to $-t$ in the original wave equation~\eqref{DWE.intro} leads us to the equation
\begin{equation}\label{timereversal}
 v_{tt}(x,t)-\frac{2\alpha}{x}v_{t}(x,t)=v_{xx}(x,t), \qquad x \in \cI:=(0,1),\quad t >0,
\end{equation}
together with boundary and initial conditions as in~\eqref{DWE.intro}. This is now a damped equation with negative damping,
for which it is not even clear under which conditions solutions will exist for small $t$. However, and as a consequence
of the results obtained in the previous sections, it is possible to see that not only are there some instances where solutions
do exist, but they will also display an interesting behavior which again is not that which is normally associated with a
linear equation.

A first remark is that when $\alpha$ is a positive integer solutions of~\eqref{timereversal} are not unique in general. Indeed, let $u(x,t)$ with $u(x,\cdot) \in C^2([0,+\infty))$ be a solution of~\eqref{DWE.intro} which goes to zero in finite time as characterized in Theorem~\ref{thm:time}.
Then any function of the form
\[
 v(x,t) = \left\{ \begin{array}{ll} 0, & 0< t \leq T\eqskip
                   u(x,T+2-t), & T < t < T+2
                  \end{array}
                  \right.
\]
is a solution of~\eqref{timereversal} on the interval $[0,T+2]$ for any positive $T$.

Consider now the case of $\alpha=1$ and assume that $v$ is a solution of~\eqref{timereversal} which is not identically zero
at time $t=0$. Then, assuming that a local solution exists in time, this cannot be extended beyond the time two horizon.
This is a direct consequence of the fact that solutions of the original equation will vanish identically after $t=2$. Hence,
if a solution $v$ existed which was defined on an interval $[0,T]$ with $T>2$, either it would have to become identically
zero for some point in that interval, or the function $u(x,t)=v(x,T-t)$ would then be a solution of~\eqref{DWE.intro}
defined on an interval of length $T$ without becoming zero on that interval. The former hypothesis is not possible, since
it would imply non-uniqueness for equation~\eqref{DWE.intro} when $\alpha$ is one, while the latter case is contradicted
by Theorem~\ref{thm:time}.

Returning to equation~\eqref{DWE.intro} and the associated eigenvalue problem, we have shown that increasing the parameter
$\alpha$ causes eigenvalues to {\it oscillate in and out of infinity}
at integer values, with the described consequences for the behavior in time of solutions of equation~\eqref{DWE.intro}. 
This raises two questions concerning the rate of decay of solutions. On the one hand, in an optimization problem
such as this, it would be natural to impose a norm restriction on the damping, say assuming that $\|a\|_{p}=a_{0}$
for some $L^{p}$ norm and a given positive number $a_{0}$, for instance. We then want to know if an optimal damping 
exists and, if so, what is its shape. The answer to these questions will be, in general, expected to depend on the
value of $a_{0}$

If, on the other hand, we are interested in the decay rate of solutions of~\eqref{DWE.intro} without imposing such
a restriction, the natural question becomes whether or not it is possible to improve upon the extinction time, that
is, the time after which all solutions vanish. More precisely, are there damping terms which will cause finite
time extinction to occur for all solutions for all time larger than some $T$ which is now smaller than $2$ (on an 
interval with unit length)?

\newpage

{\footnotesize
	\bibliographystyle{acm}

\end{document}